\documentclass[a4paper]{article}
\usepackage{setspace}
\usepackage{setspace}
\usepackage{amsthm}
\usepackage{calrsfs}
\usepackage[a4paper , lmargin = {3cm} , rmargin = {3cm} , tmargin = {2.5cm} , bmargin = {2.5cm} ]{geometry}
\usepackage{amssymb}
\author{Karim Adiprasito}
\DeclareMathAlphabet{\mathpzc}{OT1}{pzc}{m}{it}
\newenvironment{frcseries}{\fontfamily{frc}\selectfont}{}
\newtheorem*{theorem*}{theorem}
\newtheorem{theorem}{Theorem}[section]
\newtheorem{Cor}[theorem]{Corollary}
\newtheorem{lemma}[theorem]{Lemma}
\newtheorem{Problem}{Problem}

\newcommand{\conv}{\textsc{conv} }
\newcommand{\e}{\varepsilon}
\newcommand{\R}{\mathbb{R}}

\newcommand{\I}{\textrm{I} }
\newcommand{\Bs}{\mathcal{B}_{\circ}}
\newcommand{\B}{\mathcal{B}}
\renewcommand{\S}{\mathrm{S}}

\newcommand{\bd}[1]{\textsc{bd}(#1) }

\newcommand{\relintx}[1]{\textsc{relint}(#1)}
\renewcommand{\d}{\mathpzc{d} }

\renewcommand{\epsilon}{\varepsilon }
\newcommand{\dP}{\d_{PH} }
\newcommand{\card}{\textsc{card} }

\theoremstyle{remark}
\theoremstyle{definition}

\begin{document}
\title{Infinite curvature on typical convex surfaces}
\author{Karim ADIPRASITO\thanks{The final preparation of this paper was supported by the DFG within the 
research training group ``Methods for Discrete Structures'' (GRK1408). It contains one of the results of the authors diploma thesis written at TU Dortmund, Germany.}, Freie Universit\"{a}t Berlin\\ adiprasito@mi.fu-berlin.de}
\date{\today}
\maketitle
\bfseries
\mdseries
\begin{abstract} Solving a long-standing open question of Zamfirescu, we will show that typical convex surfaces contain points of infinite curvature in all tangent directions. To prove this, we use an easy curvature definition imitating the idea of Alexandrov spaces of bounded curvature, and show continuity properties for this notion.\\ 
\end{abstract}
In \cite{K}, V. Klee proved that a typical convex body has a {\it smooth} (i.e. a $C^1$-) boundary. In \cite{G}, P. Gruber added that while the boundary of a typical convex body is $C^1$, it is not $C^2$, and T. Zamfirescu complemented the picture by proving detailed theorems on the curvature of typical convex bodies (\cite{curv0, nonex}). Briefly said, following this research, we know that typical convex surfaces are  smooth, i.e. they have a $C^1$ differentiable structure, and that the curvature, albeit it does not typically exist everywhere, can take values 0 and $\infty$ only. (We refer the reader to \cite{convex} or the broader surveys \cite{bccsur} and \cite{G2} for further information.)
A question that arises naturally: is this result sharp? Do points of curvature $0$ and/or $\infty$ exist on typical convex surfaces? We will show that indeed, that is true. It is already known that the curvature typically attains the value $0$ (\cite{curv0}), but it was still a long-standing open question whether the value $\infty$ is typically attained (\cite{convex}). The main result will answer this question, thus completing the picture:
\begin{theorem}
\label{main}
On typical convex surfaces, there exists at least one umbilical point of infinite curvature.
\end{theorem}

The second section will consist of the definition of curvature bounds for convex surfaces, for which we will state a continuity property, and relate it to the classic notion of curvature.\\
In section 3 we will further develop a curvature ''indicator'' which is continuous and tells us if or if not the curvature of a convex surface is bounded with respect to aforementioned notion of curvature bounds, and state an elementary property for certain sets in Euclidean space.\\
The continuity property is one reason why we use our own curvature definition instead of the comparision of triangles used in the theory of Alexandrov spaces: Alexandrov spaces of curvature bounded above, for example, do not form a stable class with respect to Gromov-Hausdorff Topology.
Section 4 and 5 will then conclude with the proof of the central theorem of this paper.
\section{Preliminaries}
Baire Categories form an important tool in several areas of mathematics, for example functional analysis. Their use is based on Baire's theorem: 
\begin{theorem}\emph{(Baire)}
\label{baire}
Let $X$ be a complete metric space. It enjoys the following property:\\
Any subset $Y$ of $X$ which is of ''first category'' in $X$, i.e. a countable union of nowhere dense sets, has dense complement in $X$.\\
\end{theorem}
Further, any topological space $X$ which enjoys the property of theorem \ref{baire} is called a ''Baire space''. The complement of a set of first Baire category in a Baire space is called ''residual''.\\
We say that ''typical'' elements of a Baire space have property $P$ if those elements not enjoying $P$ form a set of first category.
We will from now on consider the Euclidean space $\R^d$ of some dimension $d>1$. It is known that the set $\B$ of convex bodies (compact convex subsets of our Euclidean space with nonempty interior) together with the Pompeiu-Hausdorff Metric $\dP$ forms a Baire space.\\
Furthermore, it is known that the set of strictly convex bodies $\Bs$ (that is, convex bodies whose boundary contains no straight line segment) and the set of convex bodies with smooth boundary (that is, convex bodies whose boundaries are at least $C^1$) lie residually in $\B$. This is important, because obviously:
\begin{lemma}
\label{trans}
Let $X$ be any Baire space, and $X'$ a residual subset of $X$. Then $X'$ with the induced topology from $X$ is a Baire space. If $X''$ is any residual set in $X'$, it is also a residual subset of $X$.
\end{lemma}
\bigskip
Now let us recall what is meant by curvature (see \cite{B}): Take a convex body $K$ with smooth boundary $\bd{K}$, and take a point $x$ in $\bd{K}$, $\nu$ the inward normal to $\bd{K}$ in $x$ and $\tau$ some tangential vector to $\bd{K}$ in $x$. Now let $Q$ be the set of nonnegative linear combinations of the form $x+\lambda\nu+\mu\tau$, and let $z\neq x$ be a random point in $\bd{K}\cap Q$. Define the radius $r_z$ of the circle containing $x$ and $z$ and which has center somewhere in $\{x+\lambda \nu|\lambda\geq 0\}$. We define the ''lower'' (resp. ''upper'') ''curvature radius'' in direction $\tau$:
\begin{eqnarray*}
\rho^\tau_i (x)=\liminf_{z \rightarrow x} r_z& &\rho^\tau_s (x)=\limsup_{z \rightarrow x} r_z
\end{eqnarray*}
and the ''lower'' (resp. ''upper'') ''curvature'' in direction $\tau$
\begin{eqnarray*}
\kappa^\tau_i (x)=\frac{1}{\rho^\tau_s (x)}& &\kappa^\tau_s (x)=\frac{1}{\rho^\tau_i (x)} 
\end{eqnarray*}
If the values coincide, we say the curvature ''exists in direction $\tau$'', and denote it by $\kappa^\tau(x)$. We are interested in the curvature in all directions at once, thus we define the ''lower'' (resp. ''upper'') ''curvature''
\begin{eqnarray*}
\kappa_i (x)=\inf_{\tau}\kappa^\tau_i (x) & &\kappa_s (x)=\sup_{\tau} \kappa^\tau_s(x)
\end{eqnarray*}
and say the curvature exists in $x$ iff $\kappa_i (x)=\kappa_s(x)$. In classical differential geometry, this curvature exists at ''umbilical'' points only. But for our present purpose, this notion is suitable.  
\\
The smoothness properties of convex bodies have been thoroughly studied by Klee, Gruber and Zamfirescu. An initial theorem was proved by Klee (\cite{K}), and later reproved by Gruber (\cite{G}):
\begin{theorem}\emph{(Klee)}
\label{Klee}
Typical convex surfaces (boundaries of convex bodies) are \textit{smooth}.
\end{theorem}
Investigating the second-order differential structure of $\B$, it turned out that while on typical convex surfaces the curvature exists a.e. (A. D. Alexandrov and Zamfirescu), it does not exist everywhere (Gruber and Zamfirescu). The following results have been found:\\
\begin{theorem}\emph{(Zamfirescu, \cite{curv0})}
For typical convex bodies $K$, in all boundary points $x$ of $K$, and in all tangent directions $\tau$ in $x$ at $K$, the lower 
curvature $\kappa^\tau_i (x)$ is zero, or the upper curvature $\kappa^\tau_s (x)$ is $\infty$.\\
\end{theorem}
Using a theorem due to Alexandrov, Zamfirescu concluded the remarkable
\begin{Cor}\emph{(\cite{curv0})}
\label{exist}
For typical convex bodies $K$, $\bd{K}$ is smooth and furthermore and has umbilical points of curvature 0 almost everywhere, i.e.

$$\kappa^\tau(x) = 0  \quad {\rm a.e.}$$

in all tangent directions $\tau$ at $x$.
\end{Cor}
On the other hand, Gruber found out that
\begin{theorem}\emph{(Gruber, \cite{G})}
\label{Gruber}
For typical convex bodies $K$, $\bd{K}$ is not $C^2$.
\end{theorem}
This also follows from corollary \ref{exist}. More precisely, in \cite{nonex}, Zamfirescu states that:
\begin{theorem}\emph{(Zamfirescu, \cite{nonex})}
\label{nonex}
On typical $K\in \B$, in typical points $x$ on the boundary in all tangent directions $\tau$, both the lower 
curvature $\kappa^\tau_i (x)$ is zero and the upper curvature $\kappa^\tau_s (x)$ is $\infty$.\\
\end{theorem}
Now it is immediately clear from the above that if the curvature exists somewhere on a typical convex body, it can only be 0 or $\infty$. While corollary \ref{exist} asserts that curvature $0$ indeed is typical, one is bound to ask: What about curvature $\infty$? Some efforts have been made by Zamfirescu, who was able to show some indications that indeed, $\infty$ is a typical value for the curvature of the convex body. First, corollary \ref{exist} implies the following theorem. 
\begin{theorem}\emph{(Zamfirescu, \cite{curvprop})}
\label{plane}
Typical convex bodies in the Euclidean plane contain uncountably many boundary points in which the curvature exists and is infinite.
\end{theorem}
However, for dimensions higher than 2, this theorem is not easily generalized, because the proof Zamfirescu found did not work in higher dimensions, at least not conclusively. See \cite{curvprop} for a discussion of these results.
For example, because the lower Dupin indicatrix is convex, we can infer the following theorem from theorem \ref{plane}. 
\begin{theorem}
Typical convex bodies in $\R^d$ contain uncountably many boundary points in which the curvature exists and is infinite in all tangent directions except possibly in directions lying in a $d-2$-dimensional subspace of the tangent space.
\end{theorem}
Thus, Zamfirescu asked, for example in \cite{curvprop}:
\begin{Problem}
Do typical convex bodies in Euclidean space of dimension greater than 2 possess points of existing and infinite curvature?
\end{Problem}
\section{Curvature and Cones}
In absence of smoothness, a cone can be used in an elementary fashion as an estimate the tangential cone of a surface. Without explaining how to do this, we will continue and use this idea for the issue of curvature. Alas, cones themselves are not suited for this matter. With a little adjustment, however, we will easily find more suitable sets for our needs.\\
Let $\S_{\e}(x),\ \e >0$ be the set of all points in our Euclidean space $\R^d$ with Euclidean distance $\e$ from $x\in \R^d$. Let $\tau$ be unit vector, let $\mu$ be the intrinsic metric of $\S_{\e} (x-\e\tau)$ (in the special case of the unit sphere $\S_1=\S_1(0)$ we will always use the angular metric). Additionally, let $\delta$ be some number in $(0,\frac{1}{2})$.\\ 
We define 
$$S_{\tau}(x,\epsilon,\delta)=\{y\in S_{\e}(x -\e \tau)| \mu(x,y)\leq \epsilon \delta \pi\}$$
and
$$C_{\tau}(x,\epsilon,\delta)=\bigcup_{K\in \B,\ S_{\tau}(x,\epsilon,\delta)\subset \bd{K}} K.$$
We call the above set a ''hat''. We say a subset $M$ of Euclidean Space has hat $C_{\tau}(x,\epsilon,\delta)$ iff there is a point $x\in M$ and a unit vector $\tau$ so that $M\subset C_{\tau}(x,\epsilon,\delta)$. This provides an incredibly easy way to estimate large curvature:
\begin{lemma}
Let $K\in\B$ be a convex body which has hat $C_{\tau}(x,\epsilon,\delta)$. Then $\kappa_i(x)\geq \frac{1}{\e}$, in other words: the curvature of the hat is majorized by the curvature of the convex body.
\end{lemma}
Note that the values which determine the hats are vital for this discussion. Thus we call $x$ the ''tip'', $\tau$ the ''direction'', $\e$ the ''radius'' and $\delta$ the ''angle'' of the hat.\\
The following theorem will show that the property of having a certain hat shows some continuity properties. 
\begin{theorem}
\label{cone}
Suppose $K\in\B$ has a hat $C_{\tau}(x,\epsilon,\delta)$, and let $\Delta$ be some real number smaller than $\delta$. Then there is an open neighborhood of $K$ in $\B$ so that every $K'$ in that neighborhood has hat $C_{\tau'}(x',\epsilon+\Delta,\delta-\Delta)$, where $\measuredangle(\tau',\tau)<\Delta$ and  $||x-x'||<\Delta$. In the latter inequality, $||\circ ||$ denotes the Euclidean norm.
\end{theorem}
\begin{proof}
This theorem was stated in a slightly more special form in \cite{lcurv}. Unfortunately, the form stated there misses our needs by just an inch.
Let $\phi:=\Delta \frac{1-\cos \Delta\pi}{3}$. We will show that $B(K,\phi)$, which denotes the open ball (in the Pompeiu-Hausdorff metric) around $K$ of radius $\phi$, has some of the desired properties. Let $K'\in B(K,\phi)$ be arbitrary, and set $x^+=x+\Delta\tau$.
Define
$$\alpha_0:= \sup\{\alpha| K'+\alpha\tau\subset C_{\tau}(x^+ ,\Delta+\epsilon,\delta)\} $$
and 
$$x' \in K'+\alpha_0\tau \cap \bd{C_{\tau}(x^+,\Delta+\epsilon,\delta)}.$$
Let us suppose that
$$x'\in\bd{C_{\tau}(x^+ ,\Delta+\epsilon,\delta)}\setminus \relintx{S_{\tau}(x^+, \Delta+\e,\Delta)}.$$
Now obviously $\alpha_0$: $x'-\alpha_0\tau \in K'$, and this implies
$$\d(x^+,C_{\tau}(x,\e,\delta)+B(0,\phi))=\Delta-\phi$$
which in turn implies
$$\alpha_0\geq\frac{\Delta-\phi}{\cos{\Delta\pi}}.$$
On the other hand, let $x''$ be the point in $K'$ nearest to $x$, and let $\alpha'$ be the smallest real number so that $x''+\alpha'\tau$ lies in $\bd{C_{\tau}(x^+,\rho+\e,\delta)}$.
Now by definition
$$\alpha_0\leq\alpha',$$
but also
$$\alpha'\leq \Delta+\phi.$$
Putting these together we get
$$\frac{\Delta-\Delta\frac{1-\cos{\Delta\pi}}{3}}{\cos{\Delta\pi}}\leq \Delta+\Delta\frac{1-\cos{\Delta\pi}}{3}$$
which in turn is equivalent to
$$(\cos\Delta\pi)^2-3\cos{\Delta\pi}+2\leq0.$$
But since $\Delta<\delta<\frac{1}{2}$, this inequality cannot be right.\\
Thus $x'\in\bd{C_{\tau}(x^+ ,\Delta+\epsilon,\delta)}\setminus \relintx{S_{\tau}(x^+, \Delta+\e,\Delta)}$ is wrong, which in turn yields $x'\in \relintx{S_{\tau}(x^+, \Delta+\e,\Delta)}$.
This implies that $K'$ has hat $C_{\tau'}(x'-\alpha_0 \tau, \Delta+\e,\delta-\Delta)$, where $\measuredangle(\tau',\tau)<\Delta$. Now since $||x-(x'-\alpha_0 \tau)||<\phi+\Delta(\e+\Delta)\pi$, proper adjustment of $\Delta$ gives the final inequality. \end{proof}
We will now turn to some fairly easy lemmas, which we will state without proof.
\section{Some lemmata}
First, suppose $K$ is some strictly convex body, and $\tau$ is some unit vector. Let $x_\tau$ be the unique point in which $K$ has inward normal $\tau$, and let $\e>0,\ \delta\in(0,\frac{1}{2})$ be some real numbers. Of course, $C_\tau(x_\tau,\e,\delta)$ needn't be a hat of $K$, but for some $\alpha\geq 0$, $K \subset C_\tau(x_\tau,\e,\delta)-\alpha\tau$. Let $\alpha_0$ be the smallest such number. We define a function, the ''curvature indicator'':
$$\mathfrak{K}(\epsilon,\delta,K,\tau):\mathbb{R}_{>0}\times (0,\frac{1}{2})\times \Bs\times \S_1 (0) \rightarrow \mathbb{R}_{\geq 0}$$
$$\mathfrak{K}(\epsilon,\delta,K,\tau):=\alpha_0.$$
We write $\mathfrak{K}_{\epsilon,\delta}(K,\tau)$ instead of $\mathfrak{K}(\epsilon,\delta,K,\tau)$ if we want to indicate that $\e$ and $\delta$ are constant parameters, likewise we will write $\mathfrak{K}_{\epsilon,\delta,K}(\tau)$, if $K$ is to be constant, too.
\begin{lemma}
\label{conti} The curvature indicator fulfills the following properties: 
\begin{enumerate}
\item $\mathfrak{K}(\epsilon,\delta,K,\tau)\geq0$
\item $\mathfrak{K}(\epsilon,\delta,K,\tau)=0\Leftrightarrow$ $K$ has the hat $C_{-\tau}(x_\tau,\epsilon,\delta)$
\item Let $\e>0,\ \delta\in (0,\frac{1}{2})$ be some real numbers.
$\mathfrak{K}_{\e,\delta}(K,\tau)$ is continuous as a function from $\Bs\times \S_1(0)$ to $\mathbb{R}_{\geq 0}$
\end{enumerate}
\end{lemma}
\bigskip
\bigskip
We now assert a simple geometrical lemma.\\
Let $\epsilon_i$, $\delta_i,\ i\in \I$ be real numbers, agreeing with the definition of hats. Also, for the rest of the section define $x$ to be some point in Euclidean space, and $\tau$ some unit vector.\\ 
We set $C_{\tau}(x,(\epsilon_i),(\delta_i))_\I:=\bigcap_{i\in \I} C_{\tau}(x,\epsilon_i,\delta_i)$. Now we formulate a simple lemma. 

\begin{lemma}
\label{schacht} Consider a finite set $\I$ of natural numbers (a set of indices). Let $(\epsilon_i)$,$(\epsilon'_i)$ be two sequences of real numbers, monotonically decreasing and greater than 0, which fulfill $\epsilon_i>\epsilon'_i$ for all $i \in \I$. Further let  $(\delta_i)$,$(\delta'_i)$ be monotonically decreasing sequences in $(0,\frac{1}{2})$ which fulfill $\delta'_i\geq\delta_i$ for all  $ i\in \I$.
Then the following hold: 
\begin{itemize}
\item{i)}
$$C_{\tau}(x,(\epsilon'_i),(\delta'_i))_\I\subset C_{\tau}(x,(\epsilon_i),(\delta_i))_\I.$$
\item{ii)} Let $U$ be some open neighborhood of $x$.
Then $\bd{C_{\tau}(x,(\epsilon'_i),(\delta'_i))_\I}\setminus U$ and $\bd{C_{\tau}(x,(\epsilon_i),(\delta_i))_\I}$ have positive Euclidean distance.\\
Analogously, $\bd{C_{\tau}(x,(\epsilon_i),(\delta_i))_\I}\setminus U$ and $\bd{C_{\tau}(x,(\epsilon'_i),(\delta'_i))_\I}$ have positive Euclidean distance from each other.
\end{itemize}
\end{lemma}
\section{Indicators and infinite curvature}
We will define in this section the ''indicators'' used to verify infinite curvature, and state a lemma that will imply the main theorem. The proofs are largely postponed to the next section.\\
\\
Let $(a_n)$, for the rest of this paper, be an arbitrary strictly monotonic decreasing sequence. 
For some set of indices $\I$ and some natural number $n$ define $\I(n)$ to be the $n$th element of the canonically ordered set $\I$. If no such element exists, that is, if $\I$ has less than $n$ elements, we set $\I(n)=\infty$.\\
Now, for some convex body $K$ and the sequence $(a_n)$, we define a special kind of index set, the ''maximal indicator''. As the name suggests, this indicator will give us a reasonable estimate how large something (here: the curvature) gets.
The first element of this index set, $\I_{K,(a_n)}(1)$, will be the smallest natural number $i$ such that there exist a unit vector $\tau$ and a boundary point $x$ of $K$ such that $C_\tau(x,\frac{1}{i}, a_i)$ forms a hat for $K$.\\
Now, suppose we have found out the first $m$ elements of $\I_{K,(a_n)}$. We then define $\I_{K,(a_n)}(m+1)$ as the smallest natural number $j>\I_{K,(a_n)}(m)$ such that there exist a unit vector $\tau'$ and a boundary point $x'$ of $K$ such that for all $i$ in $\I_{K,(a_n)}$ and $j$ $C_{\tau'}(x',\frac{1}{i}, a_i)$ resp. $C_{\tau'}(x',\frac{1}{j}, a_j)$ is a hat of $K$. Note that these hats have their tip and direction in common.\\
Now, we call the cardinality of the set $\I_{K,(a_n)}$ the {\it order of curvature} of $K$.
Since $(a_n)$ is fixed for our needs, we write this order as a function of $K$:
$$\mathsf{K}(K)=\card(\I_{K,(a_n)}).$$
Obviously, curvature and order of curvature seldom coincide. However, the order provides an effective way to estimate curvature. To show why this is so, consider a strictly convex body $K$ which has infinite order of curvature:\\
Since the curvature indicator is continuous (lemma \ref{conti}), the sets 
\[
\mathfrak{K}^{-1}_{\iota,\ \mu ,\ ,K}(0),\ \iota=\frac{1}{\I_{K,(a_n)}(m)},\ \mu=a_{\I_{K,(a_n)}(m)}
\]
are closed nonempty subsets of the compact unit sphere $\S_1$ for every $m\in \I_{K,(a_n)}$. By construction, the intersection of any finite number of these sets is nonempty.\\
The classic Heine-Borel Theorem now implies that any infinite intersection of these sets is nonempty, in particular, the intersection over all indices in the maximal indicator of $K$. We have proved the following:
\begin{lemma}
\label{infty}
Let $K$ be a strictly convex body with smooth boundary. 
If $\mathsf{K}(K)=\infty$, then there exists a boundary point of $K$ in which the curvature exists in every tangent direction and is infinite.
\end{lemma}
We conclude this section with an observation. Let us define the following total order:
$$\I\prec\I'\Leftrightarrow\ \exists i_0 \in \mathbb{N}\ \forall i<i_0, i\in \mathbb{N}:\I(i)=\I'(i)\wedge \I(i_0)<\I'(i_0).$$

Let $K$ be a strictly convex body. In this order $\I_{K,(a_n)}$ is the smallest index set $\I$ for which there is a unit vector $\tau$ and a boundary point $x$ of $K$ such that for all $m$ in $\I$, 
$C_\tau(x,\frac{1}{i}, a_i)$ is a hat for $K$.\\
\section{Final Steps}
Now, the following theorem seems reasonable:
\begin{lemma}
\label{final}
Typical convex bodies $K$ fulfill $\mathsf{K}(K)=\infty$
\end{lemma}
As we have seen above in lemma \ref{infty}, this immediately implies theorem \ref{main}.
Thus, we need to prove lemma \ref{final}, which in turn needs a lemma.\\
\begin{lemma}
\label{cone2}
Let $[n]:=\{1,2,3,4,...,n\}$ be given, and let $(\e_i),(\e_i'),(\delta_i),(\delta_i'),\ i\in [n]$ be monotonically decreasing positive sequences, where the sequences $(\e_i),(\e_i')$ are to be strictly monotonic. Further, let these sequences satisfy $\delta'_i<\delta_i< \frac{1}{2}$ and $\epsilon'_i>\epsilon_i$ for all $i\in [n]$. Then the following hold:
\begin{itemize}
\item{1.} Let $K\in \B$ be given so that there exist $x\in \bd{K}$ and a unit vector $\tau$ so that for all $i\in [n]$  $$C_{\tau}(x,\epsilon_i,\delta_i)$$ is a hat of $K$.\\ 
Then there exists an open neighborhood of $K$ so that for all $K'$ in this neighborhood there exist $x'\in \bd{K'}$ and a unit vector $\tau'$ so that all
 $$C_{\tau'}(x',\epsilon_i',\delta_i'),\ i\in [n]$$
are hats of $K$.
\item{2.} Let $K\in \Bs$ be a strictly convex body, so that for each unit vector $\tau$ and each $x\in \bd{K}$ there is an $i\in [n]$, so that
$$C_{\tau}(x,\epsilon_i,\delta_i)$$
is no hat of $K$. Then there exists an open neighborhood of $K$ such that every $K'$ fulfills the above property, that is: for each unit vector $\tau$ and each $x\in \bd{K}$ there is an $i\in [n]$, so that
$$C_{\tau}(x,\epsilon_i,\delta_i)$$
is no hat of $K$.
\end{itemize}
\end{lemma}
\begin{proof}
\emph{Part 1.} Let $K$ be as in the description of the lemma. Justified by lemma \ref{schacht}, we can find a $\phi>0$ so that $K+\overline{B}(0,\phi)$ has Euclidean distance at least $\phi$ from 
$$\bd{C_{\tau}(x,(\epsilon'_i),(\delta'_i))_{[n]}}\setminus \bd{C_{\tau}(x,\epsilon'_n,\delta'_n)}.$$
Now, find a $\eta>0$  with $\eta\leq\min(\delta_n-\delta_n',\e'_n-\e_n)$ such that for all unit vectors $\tau'$ with $\measuredangle(\tau,\tau')<\eta$ and all points $x'$ with $||x-x'||<\eta$
$$(K+\bar{B}(0,\phi))\cap\overline{\bd{C_{\tau'}(x',(\epsilon'_i),(\delta'_i))_{[n]}}\setminus \bd{C_{\tau'}(x',\epsilon'_n,\delta'_n)}}=\emptyset.$$
By theorem \ref{cone}, there is a $\phi'>0$ so that all $K'$  with Pompeiu-Hausdorff distance less than $\phi$ to $K$ have hat $C_{\tau'}(x',\epsilon_n+\eta,\delta_n-\eta)$ where $||x-x'||<\eta$ and $\measuredangle(\tau,\tau')<\eta$. 
With this choice of $\tau'$ and $x'$, we assert that not only $C_{\tau'}(x',\epsilon_n+\eta,\delta_n-\eta)$ is a hat, but also
$$C_{\tau'}(x',\epsilon'_n,\delta'_n).$$
Also, for any $K'\in B(K,\min(\phi,\phi'))$
$$K'\cap \overline{\bd{C_{\tau'}(x',(\epsilon'_i),(\delta'_i))_{[n]}}\setminus \bd{C_{\tau'}(x',\epsilon'_n,\delta'_n)}}=\emptyset$$
holds. But this implies 
$$K'\subset C_{\tau'}(x',(\epsilon'_i),(\delta'_i))_{[n]}$$
and thus for every $i\in[n]$,
$C_{\tau'}(x',\epsilon'_i,\delta'_i)$ is a hat of $K'\in B(K,\min(\phi,\phi'))$.\\
\emph{Part 2.}
Let $K$ be as in the assumptions of the second part of the lemma. This means, by lemma \ref{schacht}, that
$$\min_{\tau\in \S_1} \sum_{i\in [n]} \mathfrak{K}_{\epsilon_i,\delta_i} (K,\tau)>0.$$
Continuity of $\mathfrak{K}$ (lemma \ref{schacht}) and compactness of $\S_1$ imply that there is a $\phi>0$
such that for all strictly convex $K'$ in the $\phi$-ball around $K$
$$\min_{\tau\in \S_1} \sum_{i\in [n]} \mathfrak{K}_{\epsilon_i,\delta_i} (K',\tau)>0$$
which in turn, using lemma \ref{schacht}, implies the desired property.
\end{proof}
\begin{proof}[Proof of Lemma \ref{final}]
We restrict ourselves to strictly convex bodies, as justified by the transitivity of categories lemma \ref{trans}.
Let $K$ be such a strictly convex body with $\mathsf{K}(K)=m$, where $m$ is some natural number. We will show that an arbitrary open neighborhood of $K$ contains an open subset where $\mathsf{K}$ takes values larger than $m$.
Let $\tau$ be a unit vector (and $x$ the corresponding tip) coherent with the definition of the maximal indicator, that is, they are chosen so that for all $i\in \I_{K,(a_n)}$,
$$C_{\tau}(x,\frac{1}{i},a_i)$$
is a hat on $K$. Now let $B(K,\Phi)$ denote an arbitrary open ball around $K$.\\
Using lemma \ref{cone2}, Part 2. we find that there is a $\phi>0$ so that for all $K'$ in $B(K,\phi)$, the following holds with respect to the order of indices defined above:
$$\I_{K,(a_n)} \preceq \I_{K',(a_n)}.$$
Now take $\theta$ to be some real number between $0$ and $\frac{\min(\Phi,\phi)}{2}$, and define $x_\theta:=x+\tau\theta$ and $K_\theta:=\conv K\cup\{x_\theta\}$. Obviously, there is a $f\in (0,1)$ so that for all $i\in \I_{K,(a_n)}$,
$$C_{\tau}(x,\frac{f}{i},\frac{a_i}{f})$$
is a hat on $K_\theta$, and let $j>\I_{K,(a_n)}(m)$ be the smallest natural number so that 
$$C_{\tau}(x,\frac{f}{j},\frac{a_j}{f})$$
is a hat on $K_\theta$.
Now choose, justified by lemma \ref{cone2}, Part 1., a $\phi'>0$ so that for all $K'$ in $B(K_\theta,\phi')$ there exists a $x'\in \bd{K}$ and a unit vector $\tau'$ so that all
 $$C_{\tau'}(x',\epsilon_i',\delta_i'),\ i\in \I:=\I_{K,(a_n)}\cup\{j\}$$
form hats on $K'$.
Set $\phi^*=\min(\phi,\theta)$.
For all $K'\in B(K_\theta,\phi^*)$
$$\I_{K',(a_n)}\preceq \I$$
holds because $\phi^*\leq\phi$. But additionally, for these $K'$
$$\I_{K,(a_n)}\subset\I_{K',(a_n)}$$
because $\phi^*\leq\phi$.
But this implies 
$$\I_{K',(a_n)}(m+1)\leq\I(m+1)=j<\infty$$
an thus
$$\mathsf{K}(K')\geq m+1.$$.
\\
\\
Thus we have proved that for all natural numbers $m$, the set of those convex bodies with $\mathsf{K}(K)\leq m$ lies nowhere dense. This in turn proves lemma \ref{final}.
\end{proof}



\begin{thebibliography}{99}
\bibitem{lcurv}
	K. Adiprasito, T. Zamfirescu;
	{\it Large curvature on typical convex surfaces},
	to appear in Journal of Convex Analysis.
\bibitem{B} 
	H. Busemann, 
	{\it Convex Surfaces}, 
	Dover Publications,
	2008.	
\bibitem{G}
  P. Gruber;
  {\it Die meisten konvexen K\"orper sind glatt, aber nicht zu glatt},
  Math. Ann. 229 ,
  259-66,
  1977.
\bibitem{G2}
  P. Gruber;
  {\it Baire categories in convexity},
  in P. Gruber, J. Wills, (eds.)
  Handbook of Convex Geometry,
  Elsevier Science,
  1327-1346,
  1993.
\bibitem{K}
	V. Klee;
	{\it Some new results on smoothness and rotundity in normed linear spaces},
	Math. Ann. 139,
	51-63,
	1959.
\bibitem{bccsur}
  T. Zamfirescu;
  {\it Baire Categories in Convexity},
  Atti Sem. Mat. Fis. Univ. Modena, Universita di Modena,
  Vol. XXXIX,
  139-164,
  1991.  
\bibitem{curvprop}
  T. Zamfirescu;
  {\it Curvature Properties of Typical Convex Surfaces},
  Pacific Journal of Mathematics,
  Vol. 131, No. 1,
  191-207,
  1988.
\bibitem{nonex}
  T. Zamfirescu;
  {\it Nonexistence of Curvature in Most Points},
  Math. Ann.,
  Vol. 252, 
  217-219,
  1980.
\bibitem{curv0}
  T. Zamfirescu;
  {\it The Curvature of Most Convex Surfaces Vanishes Almost Everywhere},
  Math. Z.,
  Vol. 174, 
  135-139,
  1980.  
\bibitem{convex}
  T. Zamfirescu;
  \textit{The Majority in Convexity}
  Editura Universit\u{a}\c{t}ii din Bucure\c{s}ti,
  1st Volume,
  2009
\end{thebibliography}
\end{document}